\documentclass[12pt]{amsart}

\usepackage{amssymb,latexsym,amsmath,epsfig}

\usepackage{graphicx}

\def\pf{\begin{proof}}

\newtheorem{thm}{Theorem}[section]
\newtheorem{prop}[thm]{Proposition}
\newtheorem{cor}[thm]{Corollary}
\newtheorem{lem}[thm]{Lemma}

\newtheorem{remark}[thm]{Remark}

\newcommand{\bprop} {\begin{proposition}}
\newcommand{\eprop} {\end{proposition}}
\newcommand{\btheo} {\begin{theorem}}
\newcommand{\etheo} {\end{theorem}}
\newcommand{\blem} {\begin{lemma}}
\newcommand{\elem} {\end{lemma}}
\newcommand{\bcor} {\begin{corollary}}
\newcommand{\ecor} {\end{corollary}}

\newcounter{rea}
\newcounter{red}
\newcommand{\Be}{\begin{equation}}
\newcommand{\Ee}{\end{equation}}
\newcommand{\Bea}{\begin{eqnarray}}
\newcommand{\Eea}{\end{eqnarray}}
\newcommand{\Bes}{\begin{equation*}}
\newcommand{\Ees}{\end{equation*}}
\newcommand{\Beas}{\begin{eqnarray*}}
\newcommand{\Eeas}{\end{eqnarray*}}
\newcommand{\Ba}{\begin{array}}
\newcommand{\Ea}{\end{array}}

\scrollmode

\begin{document}

\title[Asymptotic approximations of Good's special functions]{Asymptotic approximations of Good's special functions arising in atomic physics}

\author[D. B\'ekoll\`e]{David B\'ekoll\`e}
\address{Department of Mathematics, Faculty of Science, University of Yaound\'e I, P.O. Box 812, Yaound\'e, Cameroon }
\email{{\tt dbekolle@gmail.com}}
\author[A. Bonami]{Aline Bonami}
\address{Institut Denis Poisson, Universit\'e d'Orl\'eans, Universit\'e de Tours \& CNRS,  B\^atiment de Math\'ematiques, Rue de Chartres, BP. 6759, 45067 Orl\'eans, France}
\email{aline.bonami@gmail.com}
\author[M. G. Kwato Njock]{Mo\"ise G. Kwato Njock}
\address{Centre For Atomic Molecular Physics and Quantum Optics (CEPAMOQ), University of Douala, Faculty of Science, P.O. Box 8580, Douala, Cameroon}
\email{{\tt mkwato@yahoo.com}}
\subjclass[2020]{}


\maketitle
\begin{abstract}
We study various asymptotic approximations of Good's special functions arising in atomic physics. These special functions are situated beyond  Anger's functions to which they are closely related. Our major tool is the method of the stationary phase.
\end{abstract}

\section{Introduction }

This article has been inspired by a series of  papers \cite{K, L}   in which the third author and his coauthors develop semi-classical methods in atomic physics. They refer to the seminal paper of R. H. Good \cite{G} and his use of WKB methods for radial wave equations as a major source of inspiration. In the calculation of dipole matrix elements for bound-bound transitions in nonhydrogenic ions, they introduced two special functions, which they called Good's functions. Namely,  
the (real-valued) Good's special functions $G_{\gamma, \rho} (x)$ and $Q_{\gamma, \xi} (x), \; \gamma\geq 0, \rho> 0, \xi>1,$ were introduced in \cite{K} and \cite{L} as
\begin{equation}\label{good}
G_{\gamma, \rho} (x):=\frac 1\pi \int_0^\pi \frac {\cos (\gamma \theta+x\sin \theta)}{\rho^2+\sin^2 \theta}d\theta \qquad \quad x\in \mathbb R
\end{equation}
and
\begin{equation}
Q_{\gamma, \xi} (x):=\frac 1\pi \int_0^\pi \frac {\cos (\gamma \theta+x\sin \theta)}{\xi-\cos \theta}d\theta \qquad \quad x\in \mathbb R.
\end{equation}

This paper will be devoted to the mathematical treatment of the first one. Moreover, we will essentially be interested in the particular case when the two parameters $\gamma$ and $x$ are equal, which has a physical meaning in terms of Rydberg states  (cf. \cite{L1}) and corresponds to strongly excited states. Some comments on the other Good's function and other choices for the relation between parameters will be given in the conclusion.

\subsection{The problem}
Let us give a specific notation for the particular case $\gamma=x$, which is the object of our study. Let
\begin{equation}\label{goodrestrict}
H(x, \rho):=G_{x, \rho}(x)=\frac 1\pi \int_0^\pi \frac {\cos x( \theta+\sin \theta)}{\rho^2+\sin^2 \theta}d\theta
\end{equation}
We study the behavior of this function, which we call the {\sl restricted Good function},  when
\begin{enumerate}
\item[1.]
$\rho \rightarrow \rho_0\in (0, \infty]$ and $x\rightarrow x_0\in [0, \infty);$
\item[2.]
$\rho \rightarrow \infty$ and $x\rightarrow \infty;$
\item[3.]
$\rho \rightarrow \rho_0\in (0, \infty)$ and $x\rightarrow \infty;$
\item[4.]
$\rho \rightarrow 0$ and $x\rightarrow \infty$ with $x\rho\rightarrow \lambda \quad (\lambda\geq 0) ;$
\item[5.]
$\rho \rightarrow 0$ and $x\rightarrow \infty$ with $x\rho\rightarrow \infty.$
\end{enumerate}
Let us recall some properties of the Good's function, which have been given in \cite{K} and \cite{L}  and justify that they call it a special function.
\begin{enumerate}
\item[(a)]
It satisfies  the differential  equation
$$\frac {d^2}{dx^2} G_{\gamma, \rho} (x)-\rho^2 G_{\gamma, \rho} (x)=-\mathbb J_\gamma (-x),$$
where $\mathbb J_\gamma$ is the closely related Anger's function defined by
$$\mathbb J_\gamma (x):=\frac 1\pi\int_0^\pi \cos (\gamma \theta-x\sin \theta)d\theta.$$
\item[(b)]
The following series expansion holds:
\begin{equation}\label{series}
G_{\gamma, \rho} (x)=\frac 1{\rho \beta}\left (\mathbb J_\gamma (-x)+\sum \limits_{k=2, 4, 6,\cdots} e^{-kt}\left (\mathbb J_{\gamma+k} (-x)+\mathbb J_{\gamma-k} (-x)   \right )   \right ),
\end{equation}
where $\beta=\sqrt {1+\rho^2}$ and $t=\arg \cosh \beta=\log (\rho+\beta).$ Replacing in (\ref{series}) the Anger's functions by their power series expansions (cf. e.g. \cite{GR})  provides a power series expansion for the Good's function, i.e. 
\begin{equation*}
G_{\gamma, \rho} (x)=\sum_{n=0}^\infty c_nx^n.
\end{equation*}
\end{enumerate}

%
Let us come back to the restricted Good function, and define the
{\sl complex restricted Good function} by
\begin{equation}\label{mathcalH} \mathcal H(x, \rho):=\frac 1\pi \int_0^\pi \frac {e^{ix( \theta+\sin \theta)}}{\rho^2+\sin^2 \theta}d\theta.
\end{equation}
Since 
$$ H(x, \rho)=\Re \; \mathcal H(x, \rho),$$
properties of $H$ may be deduced from the ones of $\mathcal H.$ In particular, it is immediate  that $\mathcal H$ is the restriction to the real line of an entire function, which can be written as the sum of its Taylor series, as well as its real part. But one can say more.
The function $\mathcal H(x, \rho)$ is well-defined for complex values of $\rho$ and $x,$ for $\Re e \rho > 0.$ Moreover, it has continuous derivatives in both complex variables, which implies that it is a holomorphic function in two complex variables on this domain. Its restriction to real values is a real-analytic function of the couple $(x, \rho)$, as well as its real part $H$. They possess derivatives at all orders.\\

Before stating our results, let us clarify the notations that we will use.
\subsection{Notations}
Let $F(x, \rho)$ be a function defined on $\mathbb R\times (0, \infty).$ We want to describe the behavior of $F$ under different conditions. Typically the function $F$ is $H$ and we are interested in its asymptotic approximation as required in the previous paragraph.  Let $h$ be another real-valued continuous function defined on $\mathbb R\times (0, \infty).$ 
\begin{enumerate}
\item
We say that $F=\mathcal O(h)$ in some open set $E$ (resp. when $(x, \rho)$ tends to $(x_0, \rho_0)$, where $(x_0,\rho_0)$ may be finite or not)   if there exists some constant $C$ such that, for $(x, \rho)\in E $ (resp.  for $(x, \rho)$ in some neighborhood of $(x_0, \rho_0)$), we have $$|F(x,\rho)|\leq Ch(x,\rho).$$
\item
We say that $F=o(h)$  when $(x, \rho)$ tends to $(x_0, \rho_0)$, where $(x_0,\rho_0)$ may be finite or not, if, for all $\varepsilon >0$, there exists some neighborhood $E_{\varepsilon}$ of $(x_0, \rho_0)$
such that $|F(x,\rho)|\leq \varepsilon h(x,\rho)$ when $(x, \rho)$ belongs to $E_\varepsilon.$ \\
\item
Let $\gamma$  be another function of $(x, \rho).$ We  say that $F=o(h)$ when $\gamma$ tends to $\infty $ if, for all
$\varepsilon >0$, there exists some constant $A$ such that $\gamma>A$ implies 
 that $|F(x,\rho)|\leq \varepsilon h(x,\rho).$
 
 This last definition is easily adapted to other asymptotic behavior of the function $\gamma.$
\end{enumerate}

In each of the situations we are interested in, our purpose is to write the asymptotic approximations of $H(x, \rho)$ in the form
$$H(x, \rho)=\mbox {(Main  Term) + (Remainder)} =M+ R,$$
with the remainder $R$ that is {\sl a priori} of smaller order compared to the main term. When effectively $R=o( |M |),$ this implies that $H(x, \rho)$ is equivalent to the main term, which we note
$$H\sim \mbox { (Main  Term)}.$$
\smallskip

The function that we denote by $R$ may change from one case to the other. Constants $C$ may also vary from one line to the other.

\subsection{Statements of results}

As we said before, $H$ has derivatives up to any order. So $H$ is continuous and has continuous derivatives. In our notations, the continuity may be written as the fact that
$$H(x, \rho)=H(x_0, \rho_0)+o(1) $$
 when $(x, \rho)$ tends  to $(x_0, \rho_0)$. Remark that this only implies that $H(x, \rho)\sim H(x_0, \rho_0)$ if $H(x_0, \rho_0)\neq 0$. In the next proposition,  we give precise estimates on $R$, which will be helpful later on. Since the function $H$ is even in the $x$ variable, we will only consider non negative values of $x$.
\begin{prop}\label{Lipnew}
For all $x, x_0, \rho, \rho_0> 0,$ the following two estimates hold

\begin{align}\label{x}
    \left \vert H(x, \rho)-H(x_0, \rho)\right \vert &\leq \min \left (\frac 1{\rho^2}, \frac \pi{2\rho}\right )\left \vert x-x_0\right \vert;\\
\left \vert H(x, \rho)-H(x, \rho_0)\right \vert &\leq \min \left (\frac {\rho+\rho_0}{\rho\rho_0}, \pi\right )\frac {\left \vert \rho-\rho_0\right \vert}{\rho_0\rho}. \label{rho}
\end{align}
\end{prop}
Next we consider the behavior of $H$ for $\rho$ tending to $\infty$ and compare the function $H$ with Anger's function. 
\begin{prop}\label{compJ}
For $\rho>\frac 4\pi,$  we have
$$H(x_0, \rho) =\frac 1{\rho^2}\mathbb J_{x_0} (-x_0)+R,$$ with $|R|\leq \rho^{-4}. $ In particular, for $x_0$ fixed, when $\mathbb J_{x_0} (-x_0)\neq 0,$
$$H(x_0, \rho)\sim \frac 1{\rho^2}\mathbb J_{x_0} (-x_0)  \quad \quad \rho \rightarrow \infty. $$
\end{prop}

Two remarks can be made.
\begin{remark}
When using also Proposition  \ref{Lipnew}, this  last statement can be generalized to $H(x, \rho)$ under the condition that $|x-x_0|\leq C/\rho^2$ for some fixed constant $C$.
\end{remark}
Other statements can be generalized in the same way. We will not give details anymore.
\begin{remark}
The function $x\mapsto\mathbb J_{x} (-x)$ has an infinity of isolated zeroes. This is an easy consequence of its asymptotic behavior: the argument is given below for the function $H(\cdot, \rho)$ (Corollary 3.7).
\end{remark}
From now on we will give the results concerning the behavior of $H$ when $x$ tends to $\infty.$ As said before, because of parity it is sufficient to assume that $x$ tends to $+\infty$.
\begin{thm}\label{main}
The following estimate holds for $x>2$.
\begin{equation}\label{est} H(x, \rho)= \frac {\Gamma \left (\frac 13 \right )}{3\pi \rho^2}\cos \pi\left (x-\frac 16\right )\left [\frac 6x\right ]^{\frac 13}+R,
\end{equation}
with $|R|\leq \frac{C}{x\rho^4}.$
In particular, when $(x,\rho)$ are such that $x\rho^3$  tend  to $\infty$, then  we have

\begin{equation} \label{sign} H(x, \rho)\sim \frac {\Gamma \left (\frac 13 \right )}{3\pi \rho^2}\cos \pi\left (x-\frac 16\right )\left [\frac 6x\right ]^{\frac 13}\end{equation}
under the condition that, moreover, $x$ stays uniformly far away from the points for which the cosine vanishes, that is $x-\frac 23\notin \mathbb Z+(-\varepsilon, +\varepsilon)$ for some positive $\varepsilon$.
\end{thm}
Let us make some comment on the second statement. Assuming \eqref{est}, we know that
$$H(x, \rho)= \frac {\Gamma \left (\frac 13 \right )}{3\pi \rho^2}\left [\frac 6x\right ]^{\frac 13}\left[\cos \pi\left (x-\frac 16\right )+O(\frac{1}{(x\rho^3)^{2/3}}\right].$$ Now the supplementary  condition on $x$ implies that the cosine is larger than some $\eta>0$, and we conclude at once for the equivalence.

This statement allows $\rho$ to tend to $\infty$ or to a finite limit, and even $\rho$ tend to $0$ as long as it does not tend to $0$ too fast compared to $x^{-1}.$

The next theorem deals  with cases where $\rho$ is small compared to $x^{-1}.$

\begin{thm}\label{main-bis}
For $x, \rho$ positive, we have
\begin{equation}\label{est-bis}
    H(x, \rho)=\frac {e^{-2x\rho}}{2\rho}+\frac 1{\pi \rho}\Re \left \{e^{-i\pi x}\int_0^\infty \frac {e^{\frac {ix\rho^3t^3}6}}{1+t^2}dt\right \}+R,
\end{equation}
with $|R|=\mathcal O(1).$ In particular, if we write $\int_0^\infty \frac {e^{\frac {i\lambda t^3}6}}{1+t^2}dt = C(\lambda) e^{i\pi \psi (\lambda)},$ we have the following.
\begin{enumerate}
\item[(i)]
Assume that when $(x, \rho)$ is such that $x$ tends to $\infty,$ $\rho$ tends to zero and $x\rho^3$ tends to a positive number $\lambda$. Then this leads to 
$$H(x, \rho)=\frac {C(\lambda)}{\pi \rho}\cos(\pi(x-\psi(\lambda)))+ o(\rho^{-1}).$$
\item[(ii)]
When $(x, \rho)$  is  such that $\rho$ tends  to zero, $x\rho$ tends to $\infty$ and $x\rho^3$ tends to zero, one has now
$$H(x, \rho)=\frac 1{2 \rho}\cos (\pi x) + o(\rho^{-1}). $$
\item[(iii)]
When $(x, \rho)$ is   such that $x$ tends to $\infty$, and $x\rho$ tends to a nonnegative number $\lambda,$ then 
$$H(x, \rho)=\frac {e^{-2\lambda}+\cos (\pi x)}{2 \rho}+ o(\rho^{-1}).$$
\end{enumerate}
\end{thm}
Remark that in each case we can write that $H$ is equivalent to the main term, but with a supplementary condition on the oscillation. For instance, in the first case, the condition may be written as:
$x-\psi(\lambda)-\frac {1}2\notin \mathbb Z+(-\epsilon, \epsilon),$ for some positive $\epsilon.$ We do no give details for the other cases.

\medskip

Let us illustrate these theorems by the particular case $x=\eta \rho^{-\alpha},$ where $\alpha$ and $\eta$ are positive numbers.

\begin{cor}\label{cor}
For $x=\eta \rho^{-\alpha},$ with $\alpha, \eta>0,$ the main term is given by 
\begin{align*}
&\frac {2\Gamma \left (\frac 13 \right )}{3\pi \rho^2}\cos \pi\left (x-\frac 16\right )\left [\frac 6x\right ]^{\frac 13} \qquad &\alpha>3;\\
&\frac {C(\eta)}{\pi \rho}\cos(x-\psi(\eta)),\qquad & \alpha=3;\\
&\frac 1{2 \rho}\cos (\pi x)\qquad & 1<\alpha<3;\\
&\frac {e^{-2\eta}+\cos (\pi x)}{2 \rho}\qquad &\alpha=1;\\
&\frac {1+\cos (\pi x)}{2 \rho}\qquad & 0<\alpha<1.
\end{align*}
\end{cor}

\

The plan of this paper is as follows. In section 2, we review some elementary facts and we prove Proposition 1.2 and Proposition 1.3. In section 3, we give a  presentation of the method of the stationary phase that will be appropriate for our goal. As an application, we obtain the asymptotic behavior of Anger's functions $\mathbb J_x (x)$ and $\mathbb J_{x+k} (-x)$ for every $k\in \mathbb Z.$ In this section, we also establish Theorem 1.4.    Theorem 1.5 is proved in section 4. Finally, in section 5, we give some hints on  the case where $\gamma=\lambda x,$  with  $\lambda \neq 1,$ as well as for the second Good's function $Q_{\gamma, \xi} (x).$

\section{Proofs of Proposition \ref{Lipnew} and Proposition \ref{compJ}}
\subsection{Proof of Proposition \ref{Lipnew} }
We gather here elementary estimates on the function $H$. We first recall that it is the restriction to $\mathbb R\times \mathbb R_+$ of a holomorphic function in two variables on $\mathbb C\times\mathbb C_+$. So it is real analytic and in particular continuous. Moreover, we have the following estimates, which proves robustness of the estimates that we develop beyond variations of the variables.
\begin{lem}\label{deriv}
For all $x, \rho> 0,$ the following estimates hold.
\begin{equation}\label{0}
|H(x, \rho)|\leq \min \left (\frac 1{\rho^2}, \frac \pi{2\rho}\right );
\end{equation}
\begin{equation}\label{x2}
|H'_x(x, \rho)|\leq \pi \min \left (\frac 1{\rho^2}, \frac \pi{2\rho}\right );
\end{equation}
\begin{equation}\label{rho2}
|H'_\rho(x, \rho)|\leq \frac 2\rho\min \left (\frac 1{\rho^2}, \frac \pi{2\rho}\right ).
\end{equation}
\end{lem}

\begin{proof}
In view of \eqref{0} we show that
\begin{equation}\label{prel}
\frac 1\pi\int_0^\pi \frac {d\theta}{\rho^2+\sin^2 \theta}\leq \min \left (\frac 1{\rho^2}, \frac \pi{2\rho} \right ).
\end{equation}
Only the second estimate deserves a proof.  We have
\begin{align*}\frac 1\pi\int_0^\pi \frac {d\theta}{\rho^2+\sin^2 \theta}&=\frac 2\pi \int_0^{\pi/2} \frac {d\theta}{\rho^2+\sin^2 \theta}\\
&\leq \frac 2\pi\int_0^{\pi /2}\frac {d\theta}{\rho^2+\frac {4\theta^2}{\pi^2}}\leq \frac 1\rho\int_0^\infty\frac {dt}{1+t^2}=\frac \pi{2\rho}.
\end{align*}
We have used the inequality $\sin \theta\geq \frac {2\theta}\pi$ and  the change of variable $t=\frac {2\theta}{\pi \rho}.$ We conclude directly for \eqref{0}. 
 For the estimate (\ref{x2}), we have
$$\left \vert H'_x(x, \rho)\right \vert\leq \frac 1\pi \int_0^\pi \frac {|\sin \theta +\theta|\,d\theta}{\rho^2+\sin^2 \theta}\leq \int_0^\pi \frac {d\theta}{\rho^2+\sin^2 \theta}.
$$
Estimate (\ref{x2}) now easily follows from (\ref{prel}), using the fact that
since $0\leq \sin \theta +\theta\leq \pi$ for all $\theta \in [0, \pi].$ 

Finally,  Estimate (\ref{rho2}) is direct:

$$\left \vert H'_\rho (x, \rho)\right \vert\leq \frac {2\rho}\pi\int_0^\pi \frac {d\theta}{\left (\rho^2+\sin^2 \theta\right )^2}\leq \frac{2}{\pi\rho}\int_0^\pi \frac {d\theta}{\rho^2+\sin^2 \theta }. $$
\end{proof}

These estimates on derivatives allow us to see that the values of $H$ change very slowly with $x$ and $\rho$. The next result is a little more precise than Proposition \ref{Lipnew} . 

\begin{cor}
For all $x, x_0\in\mathbb R$ and  $0<\rho_0<\rho,$ the following two estimates hold
\begin{align}
\left \vert H(x, \rho)-H(x_0, \rho)\right \vert &\leq \min \left (\frac 1{\rho^2}, \frac \pi{2\rho}\right )\left \vert x-x_0\right \vert;\\
\left \vert H(x, \rho)-H(x, \rho_0)\right \vert &\leq \min \left (\frac {\rho+\rho_0}{\rho\rho_0}, \pi \right )\frac {\rho-\rho_0}{\rho_0\rho}.
\end{align}
In particular, if $\rho_0<\rho <2\rho_0,$ the following estimate is valid. 
$$\left \vert H(x, \rho)-H(x, \rho_0)\right \vert \leq \pi \min\left(\frac 1{\rho_0^3}, \frac 1{\rho_0^2} \right)\left (\rho-\rho_0 \right ).$$
\end{cor}
The proof is elementary and we leave it for the reader. These estimates allow to extend to close values of $x$ and $\rho$ the estimates that we develop all along this paper. We do not give details either.

\subsection{Proof of Proposition \ref{compJ}}

The proof is direct. 

\begin{align*}
|H (x_0, \rho)-\frac 1{\rho^2}\mathbb J_{x_0} (-x_0) |
&=\frac 1{2\pi}\left|\int_{-\pi}^\pi e^{ix_0( \theta+\sin \theta)}\left (\frac 1{\rho^2+\sin^2 \theta}- \frac 1{\rho^2}\right )d\theta\right|\\
&\leq\frac 1{2\pi}\left \vert \int_{-\pi}^\pi \frac {\sin^2 \theta}{(\rho^2+\sin^2 \theta)\rho^2}d\theta\right \vert\\
&\leq \frac {1}{\rho^4}.
\end{align*}

The announced result easily follows.

\section{Stationary Phase and asymptotic behavior for $\rho$ not too small. Proof of Theorem \ref{main} }
This section is devoted to the description of the method of stationary phase, which is our main tool,  and its direct applications in our context. We follow a classical approach, as described in \cite{BH, BO, S}.
\subsection{The method of stationary phase}
  For $b>0$ (and finite) and for  $x\in \mathbb R,$ we consider the oscillatory integral of the first kind
$$I(x):=\int_0^b f(t)e^{ix\psi (t)}dt,$$
where the two  functions $f$ and $\psi$ are $\mathcal C^\infty$ on $[0, b]$ and where $\psi$ is real valued. We assume that the  first non vanishing derivative at $0$ of the phase $\psi$ is  the third one. In this particular case, we have  the following proposition, which may be seen as  a particular case of the theorems given in {\cite{BH}, p.248}, {\cite{S}, p. 334}. Here we compute the explicit constants and consider the dependence in $f$ of the remaining term.

\begin{prop} \label{prop-phase}
We suppose that
\begin{enumerate}
\item[(i)]
 $\psi(0)=\psi' (0)=\psi''(0)=\psi^{(4)}(0)=0$ and $\psi^{3} (0)=1;$
\item[(ii)]
$\psi'(t)> 0$ for all $0<t< b.$
\end{enumerate}
Then there exists a constant $C_f$ such that 

\begin{equation}\label{phase}
I(x)=\frac {e^{\frac{i\pi}{6}}}{3}\Gamma \left (\frac 13\right )\left (\frac 6 x\right )^{\frac{1}3 }f(0)
  +\frac {e^{\frac{i\pi}{3}}}{3}\Gamma \left (\frac 23\right )\left (\frac 6 x\right )^{\frac{2}3 }f'(0)
 +R  ,
\end{equation}
 where 
$$|R|\leq \frac{C_f}x \qquad\qquad x>2.$$ Moreover, there exists a constant $C$ such that
\begin{equation}\label{rest}
C_f\leq C\left(\sum_{j=0}^2\sup_{t\in[0,b]}|f^{(j)}(t)|+\int_0^b|f^{(3)}(t)|dt\right) .
\end{equation}

\end{prop}
This gives the two first terms for the asymptotic of $I(x)$  for $x$ tending to $\infty$, with a rest in $\mathcal O(\frac 1x)$. The behavior for $x$ tending to $-\infty$ is obtained by taking the complex conjugate on both sides.

\begin{proof} Ingredients of the proof can be found in the literature, but we give it for the reader's convenience. Moreover it is a little simpler in this particular case and allows us to  give the precise values of coefficients.  The first step consists of a change of variable. More precisely, we state 
\begin{equation}
\tau(t)^3:= 6\psi(t),
\end{equation}
which defines a diffeomorphism $\tau =\tau(t)$ between $[0, b)$ and $[0, \sqrt[3] {6\psi(b)})$.  Moreover, using the assumptions on the derivatives at $0$ and Taylor's Formula, we know that $6\psi(t)=t^3(1+ o (t))$  when $t$ tends to $0$, so that 
$$\tau(t)=t(1+ o (t))$$
and $\tau'(0)=1.$ Let us denote by $t(\tau)$ the inverse diffeomorphism, which is also such that $t'(0)=1$ and 
$$t(\tau)=\tau(1+ o (\tau)).$$ In particular $t''(0)=0$.
We use the change of variables defined by $t=t(\tau)$ and can write
$$I(x)=\int_0^{\tau (b)}e^{-\frac {(1-ix)\tau^3}6}g(\tau)d\tau,$$
where $g(\tau)= e^{\frac {\tau^3}6} f(t(\tau))t'(\tau).$ It is easily seen that $g(0)=f(0)$ and $g'(0)=f'(0)$.  Using Taylor's Theorem, let us write 
$$g(\tau)=f(0)+f'(0)\tau +\tau^2 R(\tau),$$
where $R(\tau)=\int_0^1(1-s)g''(\tau s)ds,$  so that $|R(\tau)|\leq \|g''\|_\infty$ and $|R'(\tau)|\leq \|g^{(3)}\|_\infty$. This allows us to  write $I(x)=f(0)I_0 (x)+f'(0) I_1 (x)+I_r (x).$
By integration by parts, using the fact that $3\tau^2 e^{a\tau^3}$ is the derivative of $a^{-1}e^{ia\tau^3}$ and the estimates on $R_r$, we find that  
$$|I_r(x)|\leq \frac{2C_f}{|1-ix|}\leq \frac{2C_f}{x}$$
with $C_f=|R(0)|+|R(\tau(b)|+\int_0^{\tau(b)}|R'(\tau)|d\tau.$
The same integration by parts can be used for integration from $\tau(b)$ to $\infty,$ so that $I_j$, for $j=0,1$ can be written as
\begin{align*}
I_j(x) &=\int_0^\infty e^{-\frac {(1-ix)\tau^3}6}\tau^j d\tau +\mathcal O(x^{-1})\\
&=\frac 13\Gamma \left (\frac{ j+1}3\right )\left (\frac 6 {1-ix}\right )^{\frac{j+1}3 }+ \mathcal O(x^{-1}).
\end{align*}
Here we use the principal branch of the functions $z^{-a}.$
With this determination, we have 
$$(1-ix)^{-\frac{j+1}3}=e^{\frac{i(j+1)\pi}{6}}x^{-\frac{j+1}3}+ \mathcal O(x^{-1})$$ for $|x|>2$. We finally recognize the two terms of the proposition. To finish the proof  we remark that up to order 3 the  derivatives  of $g$  are easily bounded in terms of the derivatives of $f.$  This gives the required estimate of $C_f$  in \eqref{rest}.
\end{proof}

\begin{remark} We have put the assumption that $\psi^{(4)}(0=0,$  which is the case in the application to Good's function. If not, there is a term to add. 
\end{remark}

\subsection{Asymptotic approximations for Anger's functions}
As a first corollary of the method of stationary phase, we recall the behavior at $\infty$ of  the function $\mathbb J_x(x)$, where $\mathbb J$ stands for Anger's function, namely
\begin{equation}
\mathbb J_x(x)\sim \frac{\sqrt 3}{6\pi}
\Gamma \left (\frac 13\right )\left (\frac 6 x\right )^{\frac{1}3 }.
\end{equation} It is sufficient for this to see that the function $\psi(t)=t-\sin t$ satisfies all the conditions required. But we are also interested in the behavior at $\infty$ of the function 
\begin{align*}\mathbb J_x(-x)&= \frac 1\pi\Re \left[\int_0^\pi  e^{-ix(t+\sin t)}dt\right] \\
&= \frac 1\pi\Re \left[ e^{-i\pi x}\int_0^\pi  e^{ix(t-\sin t )}dt \right]
\end{align*}
after having taken $\pi-t$ as a new variable. As an application of Proposition \ref{prop-phase}, it follows that 
\begin{equation}
\mathbb J_x(-x)= \frac{\Gamma \left (\frac 13\right )}{3\pi}
\left (\frac 6 x\right )^{\frac{1}3 }\cos(\pi(x-\frac 16))+ O\left (x^{-1} \right ).
\end{equation}
This function has oscillations, contrarily to the function $\mathbb J_x(x)$. The local behavior at the zeros of the cosine is given by the next term in the asymptotic development. Let us write it in a more general context. It may be useful to have asymptotic approximations for $\mathbb J_{x\pm k} (-x)$ with a control of the dependence on $k$ of the remainders. Our result, which may be of independent interest, is the following.

\begin{thm}
For $k\in \mathbb Z,$ the following estimate holds:
$$\mathbb J_{x+k} (-x)=\frac {(-1)^k}{3\pi} \left \{\Gamma \left (\frac 13 \right )\left [\frac 6x\right ]^{\frac 13}\cos \pi\left (x-\frac 16 \right )-k\Gamma \left (\frac 23 \right )\left [\frac 6x\right ]^{\frac 23}\sin \pi\left (x-\frac 13 \right )\right \}+R$$
with $|R|\leq C\frac {1+|k|^3}x$ when $ |x|>2.$  Here the constant $C$ does not depend on $x$ and $k$.
\end{thm}

\begin{remark}
Using Proposition 3 of \cite{S}, Page 334, one can prove that this function has an asymptotic development at any order, but coefficients are not explicit and there is no explicit bound of the errors in terms of $k$. 
For $k=0,$ a method to compute inductively coefficients was recently given by Lopez-Pagola \cite{LP}.
\end{remark}

\begin{proof}
This is a straightforward application of Proposition \ref{prop-phase}. Indeed, we have

\begin{align*}
     \mathbb J_{x+k} (-x)&=\frac 1\pi \Re\left[\int_{0}^\pi e^{-ix(\theta+\sin \theta)}e^{-ik\theta}d\theta\right]\\
     &=\frac {(-1)^{k}}\pi \Re\left[e^{-i\pi x}\int_{0}^\pi e^{ix(\theta-\sin \theta)}e^{ik\theta}d\theta\right].
\end{align*}So we take $\psi(t)=t-\sin t$ and $f(t)=e^{ikt}.$ The computations of derivatives are straightforward and allow to estimate the rest by using \eqref{rest}.
\end{proof}

\subsection{Asymptotic behavior of $H$ when $\rho$ is not too small}

Let us recall that $H(x, \rho)=\Re \mathcal{H}(x, \rho),$ where, by \eqref{mathcalH},
\begin{equation}
\mathcal{H}(x, \rho) =\frac {e^{-i\pi x}}\pi\int_0^\pi \frac{e^{ix\psi(t)}}{\rho^2+\sin^2(t)} dt.
\end{equation}
We can  use Proposition \ref{prop-phase} for the integral, that is, when the phase $\psi$ is given by $\psi(t)=t-\sin t$ and when the function $f$ is given by
$$f(t)=\frac{1}{\rho^2+\sin^2(t)}.$$
We claim that we have the following statement:
\begin{thm}\label{main-prep}
The following estimate holds for $x>2$.
$$\mathcal H(x, \rho)= e^{-i\pi (x-\frac 16)}\frac {\Gamma \left (\frac 13 \right )}{3\pi \rho^2}\left [\frac 6x\right ]^{\frac 13}+R,$$with $|R|\leq \frac{C}{x\rho^4}.$
In particular, when $x\rho^3$ tends to $\infty$,
$$\mathcal H(x, \rho)\sim e^{-i\pi (x-\frac 16)}\frac {\Gamma \left (\frac 13 \right )}{3\pi \rho^2}\left [\frac 6x\right ]^{\frac 13}.$$
\end{thm}
\begin{proof}
The proof is a straightforward application of Proposition \ref{prop-phase}, with $f$ that now depends on the parameter $\rho$. We have $f(0)= \rho^{-2}$ and $f'(0)=0.$ We want to know the dependence in $\rho$ of the estimate of the rest.    By symmetry it is sufficient to consider the interval $(0, \frac \pi 2).$ A direct computation gives $|f^{(j)}(t)|\leq \frac{C_j}{(\rho +t)^{j+2}}, j=0, 1, 2, 3 .$ Moreover, the integral of $(\rho+t)^{-5}$ is bounded by $\frac{1}{4\rho^4}.$
We conclude directly for the theorem.
\end{proof}
Theorem \ref{main} is a direct consequence of Theorem \ref{main-prep}. That is, the estimate \eqref{est} follows directly and we already explained why the restriction on the values of $x$ in the second statement.
\begin{remark} The condition $x\rho^3$ large is very natural. This is necessary for the main term to be smaller than the bound $\frac{\pi}{2\rho}$ given in \eqref{0}.
\end{remark}

Theorem \ref{main} allows us to study  the zeros of the function $x\mapsto H(x, \rho).$

\begin{cor}
For fixed $\rho>0,$ the function $x>0\mapsto H(x, \rho)$ has an infinite number of zeros. 
\end{cor}

\begin{proof}
For $k$  a positive integer, we note $x_k:=\frac 16+k$ so that  $\cos \pi\left (x_k-\frac 16\right )=(-1)^k.$ Moreover, it follows from \eqref{est} that $H(x, \rho)$ has same sign as $\cos (x-1/6) + R(x, \rho)$ with $|R(x, \rho)|\leq C(\rho^3x)^{-2/3}$. So, for $k>k_0(\rho),$ we have $C(\rho^3x)^{-2/3}<1$ and the quantity $H(x_k, \rho)$ has same sign as $(-1)^k.$ By continuity we know that the function has a zero between $x_k$ and $x_{k+1}.$
\end{proof}
By considering the asymptotic behavior of the derivative, it is possible to prove that for $\rho$ large enough there is also uniqueness of the zero of $H(x, \rho)$ in each of these intervals.

\section{Asymptotic behavior of $H$ when $\rho$ is small. Proof of Theorem \ref{main-bis}}

We will now consider the case when $x\rho^3$ is small, so that the classical method of stationary phase does not lead to the asymptotic behavior. In this case, the integral will be cut into two parts, from $0$ to $\frac{\pi}2$ and from $\frac \pi 2$ to $\pi$. Indeed, the phase $t+\sin t$ vanishes in $0$ and $\pi$, but at different orders, so that the treatment is not the same. 

First we prove that the denominator can be simplified with a small relative error. Namely,
we shall use the following lemma. 
\begin{lem}\label{bounded}
Let $\rho > 0.$ The real-valued function $g=g_\rho$ defined by
$$g(\theta)=\frac 1{\rho^2+\sin^2 \theta}-\frac 1{\rho^2+\theta^2}$$
is nonnegative and  bounded by $\frac{\pi^2}{12}$ on $[0, \frac \pi 2].$
\end{lem}
\begin{proof}
 We recall the inequalities $\frac {2\theta}\pi\leq\sin \theta\leq \theta$ and $\sin \theta> \theta-\frac{\theta^3}6,$ ,which are valid on the interval under consideration. It follows that,
for $0<\theta\leq \frac \pi 2,$ 
$$0\leq g(\theta)\leq \frac {\theta^2 -\sin^2 \theta}{\theta^2 \sin^2 \theta}\leq \frac {\frac {\theta^4}3}{\theta^2\times \frac {4\theta^2}{\pi^2}}=\frac {\pi^2}{12}.$$
\end{proof}
Next, once we have changed the denominator into $\rho^2 + \theta^2$, we make the change of variables $\theta = \rho t,$ and the required estimate for the  integral from $0$ to $\frac \pi 2$ is given in the  next lemma.
\begin{lem} \label{with-est} There exists a constant $C>0$  such that, for  $\rho > 0$ and $x\in\mathbb R,$  
$$\int_0^{\frac \pi{2\rho}}\frac {e^{ix\left(\rho t+\sin (\rho t)\right)}}{1+t^2}dt= \frac{\pi e^{-2|x|\rho}}{2} +R,$$
with $|R|\leq C\rho.$
\end{lem}
\begin{proof}
In this proof we take  $\tau(t):=\frac 12\left(t+\sin (t)\right).$
Then   $ \frac 12 \leq \tau'(t)\leq 1 $  and $ \frac t2 \leq \tau(t)\leq t$ on the interval $[0, \frac{\pi}{2}].$ These inequalities extend to $\tau_\rho(t):=\frac{\tau(\rho t)}{\rho}$ on the interval $[0, \frac{\pi}{2\rho}].$  Moreover, $|\tau'_\rho(t)-1|\leq \frac{\rho^2t^2}{4}.$ 
We make the change of variables $t=t(\tau_\rho)$ given by the inverse function of $\tau_\rho$ on the interval $[0, \frac \gamma \rho],$ where $\gamma: =\tau\left(\frac \pi 2\right)=\frac{\pi}4+\frac 12.$ We have 
\begin{align*}
    \int_0^{\frac \pi{2\rho}}\frac {e^{ix\left(\rho t+\sin (\rho t)\right)}}{1+t^2}dt =\int_0^{ \frac \gamma{\rho}}\frac {e^{2ix\rho\tau}}{1+t(\tau)^2}t'(\tau)d\tau
   =\int_0^{ \frac \gamma{\rho}}\frac {e^{2ix\rho\tau}}{1+\tau^2} d\tau +I_1.
\end{align*}
Let us give a bound for $I_1.$ We use the fact that $\tau\leq t(\tau)\leq 2\tau$ and $1\leq t'(\tau)\leq 2$, so that
$$\left \vert I_1\right \vert \leq C\int_0^{\frac \gamma{\rho}}\frac {|t(\tau)-\tau|+(t'(\tau) -1)(1+\tau^2)^{\frac 12}}{(1+\tau^2)^{\frac 32}} d\tau.$$
We then use the inequalities above  to have the bound  $|t'(\tau)-1|\leq \rho^2\tau^2$ and, by integration, $|t(\tau)-\tau|\leq \frac 1 3\rho^2\tau^3.$ It follows that $|I_1|\leq C\rho$ for some explicit constant $C.$
Next we write that
$$ \int_0^{\frac \gamma{\rho}}\frac {e^{2ix\rho\tau}}{1+\tau^2}d\tau = \int_0^{ \infty}\frac {e^{2ix\rho\tau}}{1+\tau^2}d\tau +I_2,$$
with 
$|I_2|\leq \int_{\frac \gamma{\rho}}^{\infty} \frac {1}{1+\tau^2}d\tau\leq\gamma \rho.$
To conclude for the proof, it  remains to see that  
$$\int_0^{ \infty}\frac {e^{2ix\rho\tau}}{1+\tau^2}d\tau =\frac{\pi e^{-2|x|\rho}}{2}.$$ But we recognize, up to the constant $\frac \pi 2,$ the Fourier transform of the Poisson kernel (or the characteristic function of the Cauchy law) at $2x\rho.$ This one is well-known, and leads to the formula given in the statement.
\end{proof}

 We can now state the following proposition:
\begin{prop}\label{summary}
There exists a constant $C>0$ such that, for  $\rho > 0$ and $x\in\mathbb R,$
\begin{equation}
  \int_0^{\frac \pi{2}}\frac {e^{ix\left( t+\sin t\right)}}{\rho^2+\sin^2 t}dt=\frac{\pi e^{-2|x|\rho}}{2 \rho} +R,
 \end{equation}
 with $|R|\leq C.$
In particular, if $\rho$ tends to $0$ and $x\rho$ tends to $\lambda$, with $\lambda$ finite, the left hand side is equivalent to  $\frac{\pi e^{-2|\lambda|}}{2 \rho}.$
\end{prop}
The next lemma is analogous to Lemma \ref{with-est}.
\begin{lem} \label{with-est2}  There exists a constant $C>0$ such that, for $\rho > 0$ and $x\in\mathbb R,$
$$\int
_0^{\frac \pi{2\rho}}\frac {e^{ix\left(\rho t-\sin (\rho t)\right)}}{1+t^2}dt= \int_0^{\infty}\frac{e^{i\frac{x\rho^3 t^3}6}}{1+t^2}  dt+R,$$
with $|R|\leq C\rho.$
\end{lem}
\begin{proof}  As in the previous section,  the function   $\tau$ is defined by $\tau^3(t) =6 (t-\sin(t))$ on the interval $[0, \frac \pi 2].$ We need extra properties of $\tau,$ which we collect now. First, from the Taylor expansion of the sinus, we get the inequality
$$t^3-\frac{t^5}{20}\leq \tau(t)^3\leq t^3.$$
It follows that $\frac{t}{2}\leq \tau(t)\leq t$ on the interval $[0, \frac \pi 2].$ Moreover, there exists  some constant $c$ such that 
\begin{equation}\label{tau}
    |\tau(t)-t|\leq c\tau(t)^3.
\end{equation} Next,if we take derivatives  the defining equation for $\tau$ and use again Taylor expansions, we get 
\begin{equation}\label{prime}
  t^2-\frac{t^4}{12}\leq \tau'(t) \tau^2\leq t^2.  
\end{equation}
It follows that $\frac{1}{2}\leq \tau'(t)\leq 4 $ on the interval $[0, \frac \pi 2].$ Finally, we claim that $|\tau'(t)-1|\leq c\tau(t)^2$ for some universal constant $c.$ Indeed, write
$$|\tau'(t)-1|\tau(t)^2\leq |\tau'(t)\tau(t)^2-t^2|+|\tau(t)^2-t^2|.$$
The first term is bounded by $t^4$ because of \eqref{prime}. One uses the fact that the ratio between $\tau$ and $t$ is bounded below and above to replace $t$ by $\tau.$ For the second term we use also this property and \eqref{tau}.

Once we have these inequalities, the proof is identical to the one of the previous lemma.
We define   $\tau_\rho(t):=\frac{\tau(\rho t)}\rho$  and $t(\tau)$ as the inverse function. The properties of the function $\tau$ translate into properties of $\tau_q$ and properties of the inverse function. We have, in the same way as before,
$$\int_0^{\frac \pi{2\rho}}\frac {e^{ix\left(\rho t-\sin (\rho t)\right)}}{1+t^2}\,dt= \int_0^{\tau_\rho \left(\frac{\pi}{2\rho} \right)}\frac{e^{i\frac {x\rho^3 \tau^3}6}}{1+t(\tau)^2}  t'(\tau) d\tau.$$
Again, the ratio of denominators stands between two constants, and $t'(\tau)$ is also contained between two constants. It suffices to consider the differences $t(\tau)-\tau$ and $t'(\tau)-1$ as in the previous lemma. From this point the proof is identical. 
\end{proof}
So we have
\begin{prop}\label{summary2}
There exists a constant $C>0$ such that, for  $\rho > 0$ and $x\in\mathbb R,$ 
\begin{equation}
 \frac 1 \pi\int_{\frac \pi{2}}^{\pi}\frac {e^{ix\left( t+\sin t\right)}}{\rho^2+\sin^2 t}dt= \frac {e^{i\pi x}}\rho \int_0^{\infty}\frac{e^{-i\frac{x\rho^3 t^3}6}}{1+t^2}  dt +R,
 \end{equation}
with $|R|\leq C.$ In particular, if $\rho$ tends to $0$ and $x\rho^3$ tends to $\lambda$, with $\lambda$ finite, 
\begin{equation}
 \frac 1 \pi\int_{\frac \pi{2}}^{\pi}\frac {e^{ix\left( t+\sin t\right)}}{\rho^2+\sin^2 t}dt\sim \frac {e^{i\pi x}}\rho \int_0^{\infty}\frac{e^{-i\frac{\lambda t^3}6}}{1+t^2}  dt.
 \end{equation}
\end{prop}

Propositions \ref{summary} and \ref{summary2} lead to the following theorem, which in turn leads to Theorem \ref{main-bis}.

\begin{thm}\label{main-pre2}
Let $x$ and $\rho$ be two positive numbers. Then
$$\mathcal H(x, \rho)=\frac {e^{-2x\rho}}{2\pi\rho}+\frac{e^{i\pi x}}{\pi \rho}\int_0^\infty \frac {e^{-\frac {ix\rho^3t^3}6}}{1+t^2}dt+R,$$
with $|R|\leq C.$ In particular, 
\begin{enumerate}
\item[(i)]
when $(x, \rho)$ is such that $x$ tends to $\infty,$ $\rho$ tends  to zero and $x\rho^3$ tends to a positive number $\lambda$, 
$$\mathcal H(x, \rho)\sim\frac{e^{i\pi x}}{\pi \rho}\int_0^\infty \frac {e^{-\frac {i\lambda t^3}6}}{1+t^2}dt;$$
\item[(ii)]
when $(x, \rho)$ is such that $x\rho^3$ tends to $0$ and $x\rho$ tends to $\infty$ 
$$\mathcal H(x, \rho)\sim\frac{e^{i\pi x}}{2 \rho};$$
\item[(iii)]
when $(x, \rho)$ are such that  $x$ tends to $\infty$ and $x\rho$ tends to a non negative number $\lambda,$  
$$\mathcal H(x, \rho)\sim\frac {e^{-2\lambda}+e^{i\pi x}}{2 \rho}.$$
\end{enumerate}
\end{thm}

Theorem \ref{main-bis} is a straightforward consequence, as well as Corollary \ref{cor}.
\medskip

\begin{remark}
The equivalence, in the case {\rm (i)}, asks for the constant given by the integral to be nonzero. Its real part is strictly positive as a consequence of the next lemma.
\end{remark}
\begin{lem}\label{cauchy}
For $\lambda>0,$
$$I(\lambda):=\int_0^\infty \frac {e^{-i \frac {\lambda u^3}6 }}{1+u^2}du=e^{-i\frac \pi 6}\int_0^\infty \frac {e^{-\frac {\lambda t^3}6}}{1+e^{-i\frac \pi 3} t^2}dt .$$
\end{lem} 
\begin{proof}[Proof of the Lemma]
We apply the Cauchy theorem to the holomorphic function $\frac {e^{-i\frac {\lambda z^3}6}}{1+z^2}$ along the closed path
$$[0, R]\dot{+} C_R \dot{+}[Re^{-\frac {i\pi}6}, 0] \quad \quad (R >0),$$
where $C_R$ is the circular arc with centre $O$ joining $R$ to $Re^{-\frac {i\pi}6}$ in the fourth quadrant. We let $R$ tend to $\infty$ and obtain the equality.
\end{proof}
The real part of $(1+e^{-i\frac \pi 3} t^2)^{-1}$ is strictly positive, which allows to conclude for the remark.

\section{concluding remarks}

In this section, we try to answer natural questions occurring from this paper.

\medskip

The first one concerns the comparison between the two methods. Are there cases when both are available? The answer is positive. Indeed, it is the case when we  assume that $\rho$ tends to $0$ and $x\rho^6$ tends to $\infty$. In this case,  \eqref{est} is meaningful since we have $x\rho^3$ tend to $\infty.$ We will see that    \eqref{est-bis} makes also sense, with the first term that tends rapidly to $0$ since $x\rho$ tends to $\infty$, the second term that behaves as the main term  of \eqref{est}, that is, has amplitude  of the order $\rho^{-2} x^{-1/3}$, and the rest that is negligible compared to this term because of the assumption $x\rho^6$ tending to $\infty.$  This relies on the following computation, which  proves that the difference between the two principal terms behaves like a rest. The following lemma gives the asymptotic behavior of  the integral $I(\lambda).$

\begin{lem}
Let $\lambda$ be a positive number. Then
$$\int_0^\infty \frac {e^{i\lambda u^3}}{1+u^2}du=\frac {e^{\frac {i\pi}6}}{3\lambda^{\frac 13}}\Gamma \left (\frac 13\right )+R,$$
with
$|R|\leq \frac {1}{3\lambda}.$
\end{lem}
\begin{proof}
We use Lemma \ref{cauchy}. Moreover, we have the inequality
$$ |(1+e^{-i\frac\pi 3}t^2)^{-1}-1|\leq t^2$$ since
$ |1+e^{-i\frac\pi 3}t^2|>1.$ So 
$$I(\lambda) =e^{\frac {i\pi}6}\int_0^\infty e^{-\lambda t^3} dt   +R,$$
with 
$$|R|\leq \int_0^\infty e^{-\lambda t^3}t^2 dt.$$
The conclusion follows at once.

\end{proof}

From this we see that, up to a rest, one gets the same main term in both methods when $x\rho^3$ tends to $\infty.$ Both methods are valid at the same time, and give of course the same result.

\medskip

The second question concerns the same problem on $G_{\gamma, \rho}(x)$ when $\gamma=\lambda x,$ with $\lambda$  a fixed nonnegative parameter distinct from $1.$ If we define the phase function by $\psi_\lambda(t) =\lambda t +\sin t, $ we see that its derivative never vanishes when $\lambda>1$, while it vanishes at one point inside $(0, \pi)$ when $\lambda<1.$ The method of the stationary phase gives a main term in $x^{-1}$ when $\lambda>1$ and a main term in $x^{-\frac 12}$ when $\lambda<1.$ It is possible to have the two parameters $x$ and $\rho$ vary, as in this paper. Computations are more technical. They will be done elsewhere.

\medskip

The last question concerns the second Good's function,
\begin{equation}
Q_{\gamma, \xi} (x):=\frac 1\pi \int_0^\pi \frac {\cos (\gamma \theta+x\sin \theta)}{\xi-\cos \theta}d\theta \qquad \quad x\in \mathbb R.
\end{equation}
We again are interested in its asymptotic  behavior when $\gamma=x.$ The main difference is that the denominator is not symmetric with respect to $\pi/2$. The method of the stationary phase can be also used  for the integral 
$$\frac 1\pi \int_0^\pi \frac {e^{ix(\theta-\sin \theta)}}{\xi+\cos \theta}d\theta \qquad \quad x\in \mathbb R.$$
The same asymptotic results as in Section 3 may be obtained, with $\rho$ replaced by $\sqrt{\xi +1}.$ Section 4 can also be adapted, or one can  use the
 following relation  between the two Good's functions. Its proof  is elementary.
\begin{prop}
The Good's functions $G_{\gamma, \rho} (x)$ and $Q_{\gamma, \xi }(x)$ are related as follows.
\begin{equation*}
Q_{\gamma, \xi} (x)=\sqrt{1+\rho^2}G_{\gamma, \rho} (x)+\frac 12\left \{G_{\gamma+1, \rho} (x)+G_{\gamma-1, \rho} (x)\right \} \quad \quad (x\in \mathbb R),
\end{equation*}
where 
\begin{equation*}
\rho=\sqrt {\xi^2-1}.
\end{equation*}
\end{prop}

\end{document}